\documentclass[12pt,UTF-8]{amsart}
\usepackage{graphicx}
\usepackage[rotateright]{rotating}
\usepackage{subfigure,fancybox}
\usepackage{amscd,amstext}
\usepackage{amsmath}
\usepackage{amssymb}
\usepackage{mathrsfs}
\usepackage{color}
\usepackage{multirow}
\usepackage{bm}
\usepackage{algorithm}
\usepackage{algorithmicx}
\usepackage{algpseudocode}                                        
\usepackage{float} 
\newtheorem{tpcl}{Tpcl}[section]
\newtheorem{theorem}[tpcl]{Theorem}

\newtheorem{definition}[tpcl]{Definition}
\newtheorem{example}[tpcl]{Example}

\begin{document}

\markboth{X. Pang and J. Wang}{A new   regularization method for Fredholm  integral equation}
\title{A new iterated Tikhonov regularization method for Fredholm  integral equation of first kind}

\author{Xiaowei Pang$^*$}
     \address{Department of Mathematics, Hebei Normal University,
     Shijiazhuang, Hebei, China.}
     \address{Hebei Key Laboratory of Computational Mathematics and Applications, Hebei,China}
     \email{pangxw21@hebtu.edu.cn}
\author{Jun Wang}
     \address{Department of Mathematics, Hebei Normal University,
     Shijiazhuang, Hebei, China.}
          \address{Hebei Research Center of the Basic Discipline Pure Mathematics,
           Hebei, China.}
     \email{wjun@hebtu.edu.cn}

\begin{abstract}

We consider Fredholm integral equation of the first kind,  present an  efficient   new iterated Tikhonov method  to solve it.  The new  Tikhonov iteration method has been proved which can achieve  the optimal order under a-priori assumption. In numerical experiments,  the new iterated Tikhonov regularization method  is compared with the classical iterated Tikhonov method, Landweber iteration method to solve the corresponding discrete problem, which indicates the validity and efficiency of the proposed method.

\end{abstract}



\maketitle

\section{Introduction}
In recent years, the research about inverse problems or ill-posed problems draw scientists' attention. It can be emerged in earth physics, engineering technology  and many other fields, such as geophysical problems \cite{SZD17}, resistivity inversion problem \cite{NGMM21}, and computed tomography \cite{PHS23}. Therefore, the investigation of ill-posed problem not only has great scientific innovation significance, but also has certain  practical importance. 

In general, the inverse problem is much more difficult to solve than the forward problem, owing to its ill-posed feature. In the mid-1960s, the regularization method for dealing with ill-posed problems proposed by Tikhonov, brought the study of ill-posed problems into a new stage. Later, Landweber rewrote the equation (\ref{eq:kxy}) into an iteration form. Afterwards, many other technologies applied to regularization method came along successively, including  precondition technique \cite{L24}, adding contraction or penalty \cite{CXSLW22}, multi-parameter regularization methods \cite{GR16}, filter based methods \cite{HR11},  methods coupling of them \cite{IJT14} or other methods \cite{M13,MK15}.  Klann etal. \cite{KR08} and Hochstenbach etal. \cite{HR11} discussed  measuring the residual error  in Tikhonov regularization with a seminorm that uses a fractional power of the
Moore-Penrose pseudo inverse of  $A^{T}A$ as weighting matrix, which lead to  fractional filter methods. The former gave the fractional Landweber method and the latter presented fractional Tikhonov method.  In \cite{HS12}, Huckle and Sedlacek also derived seminorms for the Tikhonov–Phillips regularization based on the underlying blur operator, that is using discrete smoothing-norms of the form $\|Lx\|_{2}$ to substitute the classical 2-norm $\|x\|_{2} $ for obtaining regularity, with $L$ being a discrete approximation to a derivative operator.  Using a differential operator in the Tikhonov functional , it will be smoother and get a more
accurate reconstruction. In \cite{GR16}, Gazzola and Reichel proposed  two multi-parameter regularization methods for  linear discrete ill-posed problems, which  are based on the projection of a Tikhonov-regularized problem onto Krylov subspaces
of increasing dimension. By selecting a proper set of regularization parameters and maximizing a suitable quantity, they can get the approximate solution.  Stefano etal. proposed a nested
primal–dual method for the efficient solution of regularized convex optimization problem in \cite{ABDPR24}, under a relaxed monotonicity assumption on the
scaling matrices and a shrinking condition on the extrapolation parameters, they gave the convergence result for the iteration sequence. Up to now, regularization methods are still  powerful tools to settle inverse problems. 

In this paper, our goal is to give a new iterated regularization method based on  (\ref{eq:newfilter}), for solving   Fredholm integral equation  of the first kind.  This paper is organized as follows.  In Section \ref{sec2}, we recall some basic definitions  and preliminaries about the classical Tikhonov and Landweber method, the  filter based
regularization methods and the optimal order of a regularization method. In Section \ref{sec3}, we present a new iterated Tikhonov method and give the convergence result. Some numerical examples are
reported in Section \ref{sec4}. Finally, Section \ref{sec5} gives the conclusion.

\section{Preliminaries}\label{sec2}
\quad  Fredholm integral equation  of the first kind
will be  reviewed in this section firstly. Then, 
we recall some classical results about Tikhonov method and Landweber iteration. 
\subsection{Fredholm integral equation of the first kind}\label{ssec:pre}
Many  mathematical physics inverse problems, such as the backwards heat equation problem \cite{AP99} and the image restoration problem \cite{MP23},  can be reduced to the following Fredholm integral equation  of the first kind
\begin{equation}\label{fre}
\int_{a}^{b}K(s,t)x(s)ds=y(t),~~s\in [a,b],
\end{equation} 
where $a,b$ is finite or infinite, and $x(s)$ is unknown, $K(s,t)\in C(a,b)$ is known as a kernel function. If $K(s,t)$ is a continuous kernel function,  
 (\ref{fre})  will be  written as the linear operator form
\begin{equation}\label{eq:kxy}
Kx=y.
\end{equation} 
As we all know, (\ref{fre}) is ill-posed, that is to say, at least one of the existence, uniqueness and stability of the solution is not satisfied. Here it mainly refers to instability, that is, small perturbations in the data on the right hand side will lead to infinite variations in the solution. We consider the following example to illustrate.
\begin{example}
\begin{equation*}
\int_{0}^{1}(1+ts)e^{ts}x(s)ds=y(t)=e^{t},~~0\leq t\leq 1.
\end{equation*}
This equation has the unique solution  $x(t)=1$. If we use  the Simpson’s rule to approximate the integral, and the step size $h=\frac{1}{n}$, then we can get the linear  system of equations
\begin{equation*}
\sum_{j=0}^{n}w_{j}(1+t_{i}t_{j})e^{t_{i}t_{j}}x(t_{j})=y(ih),~~i=0,1,\cdots,n,
\end{equation*}
where $w$ denotes the corresponding weight vector.  Table~\ref{one} presents the error about $x(ih)-x_{i}$ in different nodal point.
\begin{table}[!htb]
    \centering
        \caption{The error between numerical solution and true solution at  different points} \label{one}
    \begin{tabular}{||c|cccc||}
        \hline
          $t$&$n=4$&$n=8$&$n=164$&$n=32$\\
          \hline
        $0$&$-0.0774$  &$ -0.1667$ & $-4.9063$&$12$\\
        
        $\frac{1}{4}$&$1.0765$ &$ -0.4535$ & $-13.0625$&$-32$ \\
      
        $\frac{1}{2}$&$0.7730$& $-2.0363$ & $16.5000$&$13$ \\
        
        $\frac{3}{4}$&$1.0749$&$-0.4393$&$-2$&$-12$\\
       
        $1$&$0.9258$&$0.8341$&$-0.4063$&$19$\\
        \hline
    \end{tabular}

 \end{table}
 
\end{example}

From the above data, we can see that the error is not  decreasing  as the improvement of the calculation accuracy of the left integral term. It is dangerous to perform numerical calculations at this time. As we stated before, the error of the measure data  is used to be inevitable, and we can't ignore the rounding errors about the computer. So it is difficult to obtain stable numerical solutions for such problems. Based on the above reasons,  a stabled method must be adopted---regularization method.

\subsection{Tikhonov method and Landweber iteration method}
The traditional Tikhonov regularization \cite{K11}
 solves the following minimization problem 
\begin{equation}
\min_{x}J_{\alpha}^{\delta}(x):=\|Kx-y^{\delta}\|^{2}+\alpha\|x\|^{2},~~x\in X.
\end{equation}
If the operator $K:X\rightarrow Y$ is linear and bounded, the regularization $\alpha>0$, then the unique minimum $x^{\alpha,\delta}$ of $J_{\alpha}^{\delta}$ is also the unique solution of the normal equation
\begin{equation}
\left(\alpha I+K^*K\right)x^{\alpha,\delta}=K^*y.
\end{equation}
Let $(\mu_{j},x_{j},y_{j})$ be a singular system for $K$, then the solution of $Kx=y$ is presented by
\begin{equation}\label{eq:picard solution}
x=\sum\limits_{j=1}^{\infty}\frac{1}{\mu_{j}}(y,y_{j})x_{j}.
\end{equation}
By the way, Tikhonov method gave a strategy 
\begin{equation}\label{eq:Tik filter}
q(\alpha,\mu)=\frac{\mu^{2}}{\alpha+\mu^{2}}
\end{equation}
to damp the factor $\frac{1}{\mu_{j}}$ of (\ref{eq:picard solution}). The function $q(\alpha,\mu):(0,\infty)\times(0,\|K\|]\rightarrow \mathbb{R}$ is called as a regularizing filter function. Based on these information, Tikhonov regularization strategy $R_{\alpha}:Y\rightarrow X,\alpha>0$ is defined by 
\begin{equation}
R_{\alpha}y=\sum\limits_{j=1}^{\infty}\frac{1}{\mu_{j}}\frac{\mu_{j}^{2}}{\alpha+\mu_{j}^{2}}(y,y_{j})x_{j},~~~y\in Y.
\end{equation}
Another methodology to give a regularizing filter function is
\begin{equation}\label{eq:Lan filter}
q(\alpha,\mu)=1-(1-a\mu^{2})^{\frac{1}{\alpha}},
\end{equation} 
if $\frac{1}{\alpha}=m$, and $m$ represents the iterations, then (\ref{eq:Lan filter}) is the filter function of the Landweber iteration 
\begin{equation}\label{eq:Land}
x^{0}:=0,~~x^{m}=(I-aK^*K)x^{m-1}+aK^*y.
\end{equation}
(\ref{eq:Tik filter}) and (\ref{eq:Lan filter}) are regularization filter functions, both of them satisfy the following definition.
\begin{definition}\cite{BBDS15}
Let $K:X\rightarrow Y$ be compact with singular system $(\mu_{j},x_{i},y_{j})$,  $\mu(K)$  be the closure of $\bigcup\limits_{j=1}^{\infty}\{\mu_{j}\}$,
and $q:\mu(K)\subset(0,\mu_{1})\rightarrow\mathbb{R}$ be a function  with the following properties:
\begin{subequations}
\begin{align}
&\sup_{\mu_{j}>0}|\frac{q(\alpha,\mu_{j})}{\mu_{j}}|=c(\alpha)<\infty,\label{eq:A1}\\
&|q(\alpha,\mu_{j})|\leq c<\infty,~~\text{c is independent of $\alpha, j$},\label{eq:A2}\\
&\lim_{\alpha\rightarrow 0}q(\alpha,\mu_{j})=1~\text{pointwise in $\mu_{j}$}\label{eq:A3}.
\end{align}
\end{subequations}

 Let $R_{\alpha}:Y\rightarrow X$ be a family operators, $\alpha>0$, which is defined by
\begin{equation}
R_{\alpha}y=\sum\limits_{j=1}^{\infty}\frac{q(\alpha,\mu_{j})}{\mu_{j}}(y,y_{j})x_{j},~~~y\in Y,
\end{equation}
then  it is a regularization strategy or a filter based regularization method with $\|R_{\alpha}\|=c(\alpha)$, and $q(\alpha,\mu)$ is called a filter function.
\end{definition}
In addition, regularization operators corresponding to (\ref{eq:Tik filter}) and (\ref{eq:Lan filter}) are optimal order under an a-priori assumption \cite{BD17}, or are optimal strategies in the sense of the worst-case error \cite{K11}. For  the integrity of the article, we give the definition. Next to the definition, there is a sufficient theorem to realize the optimal convergence rate. 
\begin{definition}
For given $\sigma,E>0$, let 
\begin{equation*}
X_{\sigma,E}:=\{x\in X|~ \exists~z\in X,~\|z\|\leq E,~x=\left(K^*K\right)^{\frac{\sigma}{2}}z\}\subset X.
\end{equation*}
Define 
\begin{equation*}
\mathcal{F}(\delta,\sigma,R_{\alpha}):=sup\{\|x-x^{\alpha,\delta}\|:~x\in X_{1},~\|y-y^{\delta}\|\leq \delta\},
\end{equation*}
for any $X_{1}\subset X$ a subspace, $\delta>0$, and for a regularization method $R_{\alpha}$, if
\begin{equation*}\label{eq:worstcaseerr}
\mathcal{F}(\delta,\sigma,R_{\alpha})\leq c\delta^{\frac{\sigma}{\sigma+1}}E^{\frac{1}{\sigma+1}}
\end{equation*}
holds, then a regularization method $R_{\alpha}$ is called of optimal order under the a-priori assumption $x\in X_{\sigma,E}$.
If $E$ is unknown, then redefine a set
\begin{equation*}
X_{\sigma}:=\bigcup_{\sigma>0}X_{\sigma,E},
\end{equation*}
and if  
\begin{equation*}
\mathcal{F}(\delta,\sigma,R_{\alpha})\leq c\delta^{\frac{\sigma}{\sigma+1}},
\end{equation*}
holds, then we call a regularization method $R_{\alpha}$ is of optimal order under the a-priori assumption $x\in X_{\sigma}$.
\end{definition} 
\begin{theorem}\cite{L89}
Let $K:X\rightarrow Y$  be a linear compact operator, $R_{\alpha}:Y\rightarrow X$ is a filter based regularization method, it will be of optimal order under the a-priori assumption $x\in X_{\sigma,E}$, $\sigma,E>0$, 
\begin{subequations}
\begin{align}
\label{eq:B1}&\sup_{0<\mu\leq \mu_{1}}|\frac{q(\alpha,\mu)}{\mu}|\leq c\alpha^{-\gamma},\\ 
\label{eq:B2}&\sup_{0<\mu\leq\mu_{1}}|(1-q(\alpha,\mu))\mu^{\sigma}|\leq c_{\sigma}\alpha^{\gamma\sigma},
\end{align}
\end{subequations}
with the regularization parameter $\alpha=\hat{c}\left(\frac{\delta}{E}\right)^{\frac{1}{\gamma(\sigma+1)}}$, $\hat{c}=\left(\frac{c}{\sigma c_{\sigma}}\right)^{\frac{1}{\gamma(\sigma+1)}}>0$. 
\end{theorem}

\section{Iterated Tikhonov regularization method}\label{sec3}

\quad In this section, we first look back the standard  Tikhonov iteration method, then introduce  a  new iterated Tikhonov regularization method, it is a generalization of the classical Tikhonov method.

The standard  Tikhonov iteration method is 
\begin{equation}\label{eq:iterTik}
x^{0,\alpha,\delta}=0,~~	\left(\alpha I+K^*K \right)x^{m+1,\alpha,\delta}=K^*y^{\delta}+\alpha x^{m,\alpha,\delta}.
\end{equation}
It can be shown that the corresponding regularization filter function is
\begin{equation}
	q^{m}(\alpha,\mu)=1-\left(\frac{\alpha}{\alpha+\mu^2}\right)^m,~~m=1,2,...
\end{equation}
	In \cite{BD17,HS12}, the authors introduced a Weighted-II Tikhonov method as the filter based method with the filter function  
	\begin{equation}
	q_{l}(\alpha,\mu)=\frac{\mu^{2}}{\mu^2+\alpha\left(1-\left(\frac{\mu}{\mu_{1}}\right)^{2}\right)^{l}},
	\end{equation}
	for $\alpha>0$ and $l\in \mathbb{N}$.
	Here, we also recall  a filter based method---the fractional Tikhonov method \cite{BBDS15,HR11} with filter function
	\begin{equation}
	q^{r}(\alpha,\mu)=\frac{\mu^{2r}}{\left(\alpha+\mu^{2}\right)^{r}},
	\end{equation}
	for $\alpha>0$ and $r\geq \frac{1}{2}$.
	Now, we can introduce a mixed method which combines the filter function of the fractional Tikhonov method and weighted-II Tikhonov method.
	\begin{definition}\label{def:Rlr}
	Fixing $q_{l}^{r}(\alpha,\mu)$ such that
	\begin{equation}\label{eq:newfilter}
	q^{r}_{l}(\alpha,\mu):=\frac{\mu^{2r}}{\left(\alpha\left(1-\left(\frac{\mu}{\mu_{1}}\right)^{2}\right)^{l}+\mu^{2}\right)^{r}},
	\end{equation}  we define the mixed method (fractional weighted Tikhonov method) as the filter based method
	\begin{equation}\label{eq:new Ra}
	R_{\alpha,l}^{r}y:=\sum\limits_{j=1}^{\infty}\frac{q_{l}^{r}(\alpha,\mu)}{\mu_{j}}(y,y_{j})x_{j}.
	\end{equation}
	\end{definition}
	It is clear that for $l=0$ and $r=1$, it becomes the classical Tikhonov method.
	\begin{theorem}
	Let $K:X\rightarrow Y$ be a linear compact operator with infinite dimensional range and $R_{\alpha,l}^{r}$ be the corresponding family mixed method operator. Then for every given $r\geq \frac{1}{2}$, $l\in \mathbb{N}$, $R_{\alpha,l}^{r}$ is a regularization method of optimal order under the a-priori assumption $x\in X_{\sigma,E}$ with $0<\sigma\leq 2$. Further, if the regularization parameter satisfies $\alpha=(\frac{\delta}{E})^{\frac{2}{\sigma+1}}$, then the best possible rate of convergence with respect to $\delta$ is $\|x-x^{\alpha,\delta,r}\|=\mathcal{O}(\delta^{\frac{2}{3}})$ with $\sigma=2$. Moreover, if $\|x-x^{\alpha,\delta}\|=\mathcal{O}(\alpha)$, then $x\in X_{2}$.
	\end{theorem}
\begin{proof}
It has been  proved $q_{l}(\alpha,\mu)$ is a filter function in \cite{BD17}, so  $q_{l}(\alpha,\mu)$ satisfies the filter function conditions (\ref{eq:A1}-\ref{eq:A3}). By Proposition 12 of \cite{BBDS15}, for $r\geq \frac{1}{2}$, the function $q_{l}^{r}(\alpha,\mu)$ which meets the condition (\ref{eq:A1}-\ref{eq:A3}) can be  verified  easily. The proof of $Q_{l}^{r}(\alpha,\mu)$ can meet the requirements (\ref{eq:B1}-\ref{eq:B2}) combining the proof of Proposition 12 in \cite{BBDS15}. The difference is that $Q_{\alpha}^{1}(\alpha,\mu)$ is the weighted-II Tikhonov method, and it also has the optimal order $\mathcal{O}(\delta^{\frac{2}{3}})$ with $\gamma=\frac{1}{2}$ in (\ref{eq:B2}) for every $0<\sigma\leq 2$.
\end{proof}
	
	In the following,  we will discuss the saturation for the mixed Tikhonov regularization.
	\begin{theorem}
	Let $K:X\rightarrow Y$ be a linear compact operator with  infinite dimensional range and let $R_{\alpha,l}^{r}$ be the  corresponding
	family of fractional Tikhonov regularization operators in Definition \ref{def:Rlr} with $r\geq \frac{1}{2}$, $l\in \mathbb{N}$. Let $\alpha=\alpha(\delta,y^{\delta})$ be any parameter choice rule, and if 
	\begin{equation}
	\sup\{\|x-x^{\alpha,\delta,r}\|:\|P(y-y^{\delta})\|\leq\delta\}=o\left(\delta^{\frac{2}{3}}\right), 
	\end{equation}
	then $x=0$ with $P$ is the orthogonal projector onto $\overline{R(K)}$.
		\end{theorem}
\begin{proof}
For $r=1$, it is clear that the saturation result follows from  weight-II  Tikhonov regularization \cite{BBDS15}. For $r\neq 1$, we have
\begin{equation}
x-x_{r}^{\alpha,\delta}=\sum\limits_{j=1}^{\infty}\frac{1}{\mu_{j}}\left(1-q_{l}^{r}(\alpha,\mu_{j})\right)(y,y_{j})x_{j},
\end{equation}
and 
\begin{eqnarray*}
\begin{aligned}
1-q_{l}^{r}(\alpha,\mu)
&=\frac{\left(1+\frac{\mu^{2}}{\alpha\left(1-\left(\frac{\mu}{\mu_{1}}\right)^{2}\right)^{l}}\right)^{r}-\left(\frac{\mu^{2}}{\alpha\left(1-\left(\frac{\mu}{\mu_{1}}\right)^{2}\right)^{l}}\right)^{r}}{\left(1+\frac{\mu^{2}}{\alpha\left(1-\left(\frac{\mu}{\mu_{1}}\right)^{2}\right)^{l}}\right)^{r-1}}\cdot\left(1-q_{l}^{1}(\alpha,\mu)\right).
\end{aligned}
\end{eqnarray*}
We notice that the above equality will be 
\begin{equation*}
1-q_{l}^{r}(\alpha,\mu)=f\left(\frac{\mu^{2}}{a}\right)\cdot\left(1-q_{l}^{1}(\alpha,\mu)\right),
\end{equation*}
where  $f(x)=\frac{(x+1)^{r}-x^{r}}{(x+1)^{r-1}}$,  $a=\alpha\left(1-\left(\frac{\mu}{\mu_{1}}\right)^{2}\right)^{l}$ and $q_{l}^{1}(\alpha,\mu)=q_{l}(\alpha,\mu)$ is the weighted-II Tikhonov. $f(x)$ is a monotone function, it satisfies $f(0)=1$ and $\lim\limits_{x\rightarrow \infty}f(x)=r$. Hence, we can get
\begin{equation*}
\min\{1,r\}\left(1-q_{l}^{1}(\alpha,\mu)\right)\leq\left(1-q_{l}^{r}(\alpha,\mu)\right)\leq \max\{1,r\}\left(1-q_{l}^{1}(\alpha,\mu)\right).
\end{equation*}
Naturally, 
\begin{equation*}
\sup\{\|x-x^{\alpha,\delta,r}\|:\|P(y-y^{\delta})\|\leq\delta\}\geq\min\{1,r\}\cdot\sup\{\|x-x^{\alpha,\delta,1}\|:\|P(y-y^{\delta})\|\leq\delta\},
\end{equation*}
for every $y^{\delta}$ satisfies $\|y-y^{\delta}\|\leq \delta$. From Proposition 3.6 in \cite{BBDS15}, we have
\begin{equation*} \sup\{\|x-x^{\alpha,\delta,1}\|:\|P(y-y^{\delta})\|\leq\delta\} =o(\delta^{\frac{2}{3}}),
\end{equation*}
Hence, the conclusion follows  from the saturation result for Weighted-II Tikhonov (see Corollary 5.3 \cite{BBDS15}). 
\end{proof}
	From now on, we propose a  new iterated regularization method based on the above mixed Tikhonov. By iterations, we find that a large $m$ will provide a faster convergence rate (see Theorem \ref{thm:lrrate}).
	\begin{definition}
	Define the iterated fractional weighted Tikhonov method as

\begin{equation}\label{eq:newiter}
\resizebox{0.9\hsize}{!}{$
\left(K^*K+\alpha\left(I-\frac{K^*K}{\|K^*K\|}\right)^{l}\right)^{r}x^{m,\alpha}=(K^*K)^{r-1}K^*y
	+\left[\left(K^*K+\alpha\left(I-\frac{K^*K}{\|K^*K\|}\right)^{l}\right)^{r}-(K^*K)^{r}\right]x^{m-1,\alpha}	
	$}
\end{equation}
with $	x^{0,\alpha}:=0,~r\geq \frac{1}{2},\alpha>0$ and $l\in \mathbb{N}$.
	We can define $x^{m,\alpha,\delta}$ as the $m$-th iteration of (\ref{eq:newiter}) whenever $y$ is replaced by the noise data $y^{\delta}$.
	\end{definition}
	In whole paper, for convenience, (\ref{eq:newiter}) will be  called the new iterated Tikhonov method. 
	\begin{theorem}\label{thm:lrrate}
	The  new iterated Tikhonov  in (\ref{eq:newiter}) is a filter based regularization method with filter function
	\begin{equation}
	Q^{m,r}_{l}(\alpha,\mu)=1-\left(1-Q^{r}_{l}(\alpha,\mu)\right)^{m},
	\end{equation}
	with $Q^{r}_{l}(\alpha,\mu)=q_{l}^{r}(\alpha,\mu)=\left(\frac{\mu^{2}}{\mu^{2}+\alpha(1-(\frac{\mu}{\mu_{1}})^{2})^{l}}\right)^{r}$.
	Moreover, this method is of optimal order under the a-priori assumption $x\in X_{\sigma,E}$, for $l\in \mathbb{N}$ and $0\leq \sigma\leq 2m$. Further,  regularization parameter  $\alpha=\left(\frac{\delta}{E}\right)^{\frac{2m}{1+\sigma}}$,  yields the best convergence rate $\|x-x^{m,\alpha,\delta}\|\leq\mathcal{O}(\delta^{\frac{2m}{2m+1}})$ with $\sigma=2m$.
\end{theorem}	
\begin{proof}
	Denote $C=\left(K^*K+\alpha\left(I-\frac{K^*K}{\|K^*K\|}\right)^{l}\right)^{r}$, $B=C^{-1}(K^*K)^{r-1}K^*y$, and $A=C^{-1}\left[C-(K^*K)^{r}\right]$. By the iteration formulas (\ref{eq:newiter}), we have 
	\begin{eqnarray*}
	\begin{aligned}
	x^{m,\alpha,\delta}&=Ax^{m-1,\alpha,\delta}+B=A^{2}x^{m-2,\alpha,\delta}+(A^{1}+A^{0})B\\
	&=\cdot\cdot\cdot\\
	&=\sum\limits_{k=0}^{m-1}A^{k}B
	=\sum\limits_{k=0}^{m-1}C^{-k}\left[C-(K^*K)^{r}\right]^{k}C^{-1}(K^*K)^{r-1}K^*y.
	\end{aligned}
	\end{eqnarray*}
	Let $R_{\alpha}^{m}=\sum\limits_{k=0}^{m-1}C^{-k}\left[C-(K^*K)^{r}\right]^{k}C^{-1}(K^*K)^{r-1}K^*$, and  the singular system  be $\{\mu_{j},x_{j},y_{j}\}$, then 
	\begin{eqnarray*}
	\begin{split}
	\hspace{-5mm}R_{\alpha}^{m}y
	&=\sum\limits_{j=1}^{\infty}\frac{1}{\mu_{j}}Q_{l}^{m,r}(\alpha,\mu_{j})(y,y_{j})x_{j},
	\end{split}
	\end{eqnarray*}
	and 
	\begin{scriptsize}
	\begin{eqnarray*}
	\begin{split}
	\hspace{-5mm}Q_{l}^{m,r}(\alpha,\mu_{j})=\sum\limits_{k=0}^{m-1}\left(\frac{\left(\mu_{j}^{2}+\alpha\left(1-(\frac{\mu_{j}}{\mu_{1}})^{2}\right)^{l}\right)^{r}-\mu_{j}^{2r}}{\left(\mu_{j}^{2}+\alpha\left(1-(\frac{\mu_{j}}{\mu_{1}})^{2}\right)^{l}\right)^{r}}\right)^{k}\left(\frac{\mu_{j}^{2}}{\left(\mu_{j}^{2}+\alpha\left(1-\left(\frac{\mu_{j}}{\mu_{1}}\right)^{2}\right)^{l}\right)}\right)^{r}.
	\end{split}
	\end{eqnarray*}
	\end{scriptsize}
	By the definition of  $q_{l}(\alpha,\mu)$
,  then
	\begin{equation*}
	Q_{l}^{m,r}(\alpha,\mu)=\sum\limits_{k=0}^{m-1}\left(1-\left(q_{l}(\alpha,\mu)\right)^{r}\right)^{k}\left(q_{l}(\alpha,\mu)\right)^{r}.
	\end{equation*}
	It is easily to get $Q_{l}^{m,r}(\alpha,\mu)=1-\left(1-(q_{l}(\alpha,\mu))^{r}\right)^{m}$, that is the conclusion as we stated.   

From  the relationship between $Q_{l}^{m,r}(\alpha,\mu)$ and $Q_{l}^{r}(\alpha,\mu)$, we can deduce 
\begin{equation*}
Q_{l}^{r}(\alpha,\mu)\leq Q_{l}^{m,r}(\alpha,\mu)\leq mQ_{l}^{r}(\alpha,\mu).
\end{equation*}
	Clearly, $q_{l}(\alpha,\mu)$ is weighted-II Tikhonov and it is a regularization filter method. Hence,  $Q_{l}^{m,r}(\alpha,\mu)$ satisfies  the conditions (\ref{eq:A1}-\ref{eq:A2}) and (\ref{eq:B1}). At the same time, (\ref{eq:A3}) is easy to check,  so it is a filter function naturally. 
	 $Q_{l}^{m,r}(\alpha,\mu)$ adapt to the filter based regularization conditions. Finally, we make sure $Q_{l}^{m,r}(\alpha,\mu)$ satisfies condition (\ref{eq:B2}) for the order optimality.
\begin{eqnarray*}
\begin{aligned}
1-Q_{l}^{m,r}(\alpha,\mu)&=\left(1-Q_{l}^{r}(\alpha,\mu)\right)^{m}
\leq 1-Q_{l}^{r}(\alpha,\mu)\\
&\leq \max\{1,r\}^{m}(1-Q_{l}^{1}(\alpha,\mu))^{m}
=c\left(1-Q_{l}^{m,1}(\alpha,\mu)\right),
\end{aligned}
\end{eqnarray*}
and notice that $Q_{l}^{1}(\alpha,\mu)=q_{l}(\alpha,\mu)=\frac{\mu^{2}}{\mu^{2}+\alpha\left(1-(\frac{\mu}{\mu_{1}})^{2}\right)^{l}}$
  is the weighted-II filter function and $Q_{l}^{m,1}(\alpha,\mu)=1-\left(1-\frac{\mu^{2}}{\alpha\left(1-\left(\frac{\mu}{\mu_{1}}\right)^{2}\right)^{l}+\mu^{2}}\right)^{m}$ is the filter function  of the stationary iterated Tikhonov. So that condition (\ref{eq:B2}) follows from the properties of stationary Weighted-II iterated Tikhonov,  and $\gamma=\frac{1}{2}$, $0\leq \sigma\leq 2m$, therefore, we get the best convergence rate $\mathcal{O}(\delta^{\frac{2m}{2m+1}})$.
  \end{proof}

	 \section{Numerical experiments}\label{sec4}
	 The purpose of this  section is to illustrate the validity from the previous sections 
	 with the following  example.  The classical iterative Tikhonov regularization method,  Landweber method and the new iterated Tikhonov method are adopted to get the iterative numerical solutions.
	 
	 Consider the following
	 integral equation of the first kind:
	 \begin{equation}\label{eq:example}
	 	\int_{0}^{\infty}e^{-st}x(t)dt=h(s),~~0\leq t<\infty.
	 	\end{equation} 	
	 The kernel operator is given by $(Kx)(t)=\int_{0}^{\infty}e^{-st}x(s)ds$. For numerical computation, we use Gauss-Laguerre quadrature rule with $n$ points to get the matrix $A$ corresponding to the kernel. 
	  The measure data about the right-hand side function is denoted by $y^{\delta}=y+\delta\|\eta\|$,   where $\eta$ obeys the standard normal distribution, and the perturbation magnitude is $\delta$.
	 
	 In this example, the right-hand side function $h(s)=\frac{2}{2s+1}$, hence (\ref{eq:example}) has the unique solution $x(t)=e^{-\frac{t}{2}}$.  As mentioned above, we can use regularization method to solve the numerical solution.  The  operator $K$ is  self-adjoint, so discrete Tikhonov equation takes the following form
	 \begin{equation}\label{eq:Fredholm1}
	 	\left(\alpha I+A^{2}\right)x^{\alpha,\delta}=Ay^{\delta}.
	 \end{equation}
	 
	 \subsection{Classical method implementation}
	  First, let the perturbation $\delta=0$, that is only the discrete error by quadrature rule will be generated,  choose different  regularization parameter $\alpha=10^{-i},i=1,2,...,10$ by priori,  and  the quadrature points number $n=16,32$. The numerical discrete errors variation diagram $|x-x^{\alpha,\delta}|_{l^{2}}$ are showed in Figure~\ref{fig:alpha_error_delta0}.
	 From Figure~\ref{fig:alpha_error_delta0}, if $\alpha$ is small,   the error  has a big difference between $n=16$ and $n=32$ as $\alpha<10^{-4}$ especially.
\begin{figure}[htbp]
\centering
\subfigure[Numerical error for different regularization parameter $n$ in Tikhonov]{\label{fig:alpha_error_delta0}
\includegraphics[width=0.45\textwidth]{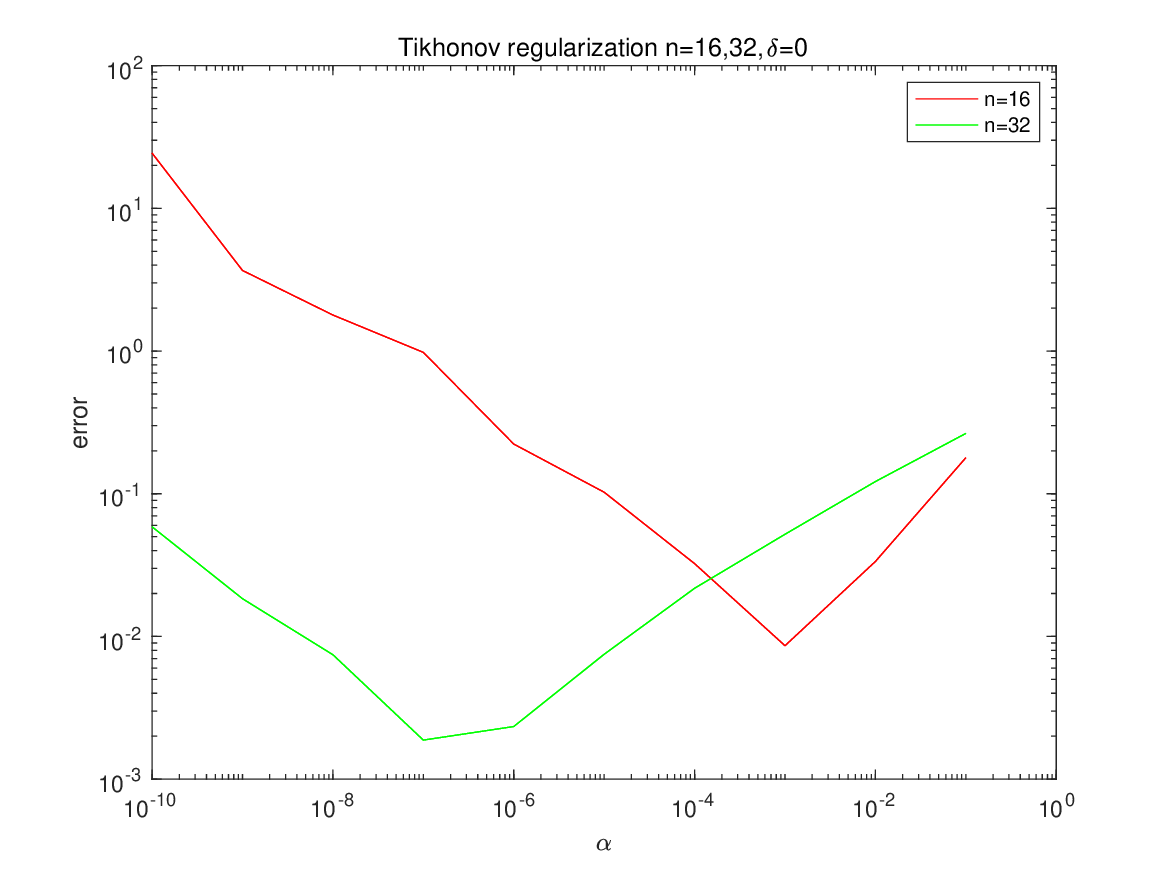}}\hfill
\subfigure[Numerical error for different perturbation parameter $\delta$ in Tikhonov]{\label{fig:alpha_error_n32_one}
\includegraphics[width=0.45\textwidth]{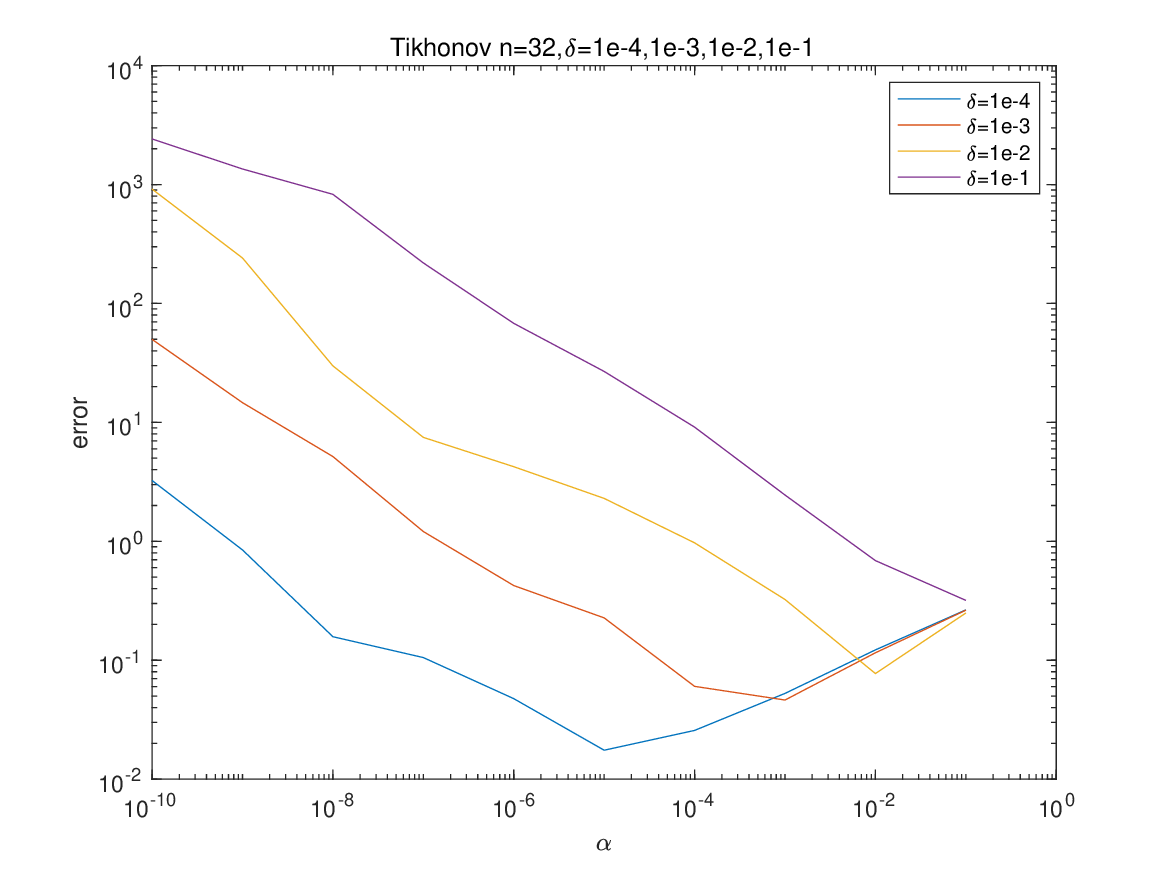}}\hfill
	 	\subfigure[Numerical error for different perturbation parameter $\delta$ in Landweber]{\label{fig:Landerr_m_n32_one}
	 	\includegraphics[width=0.45\textwidth]{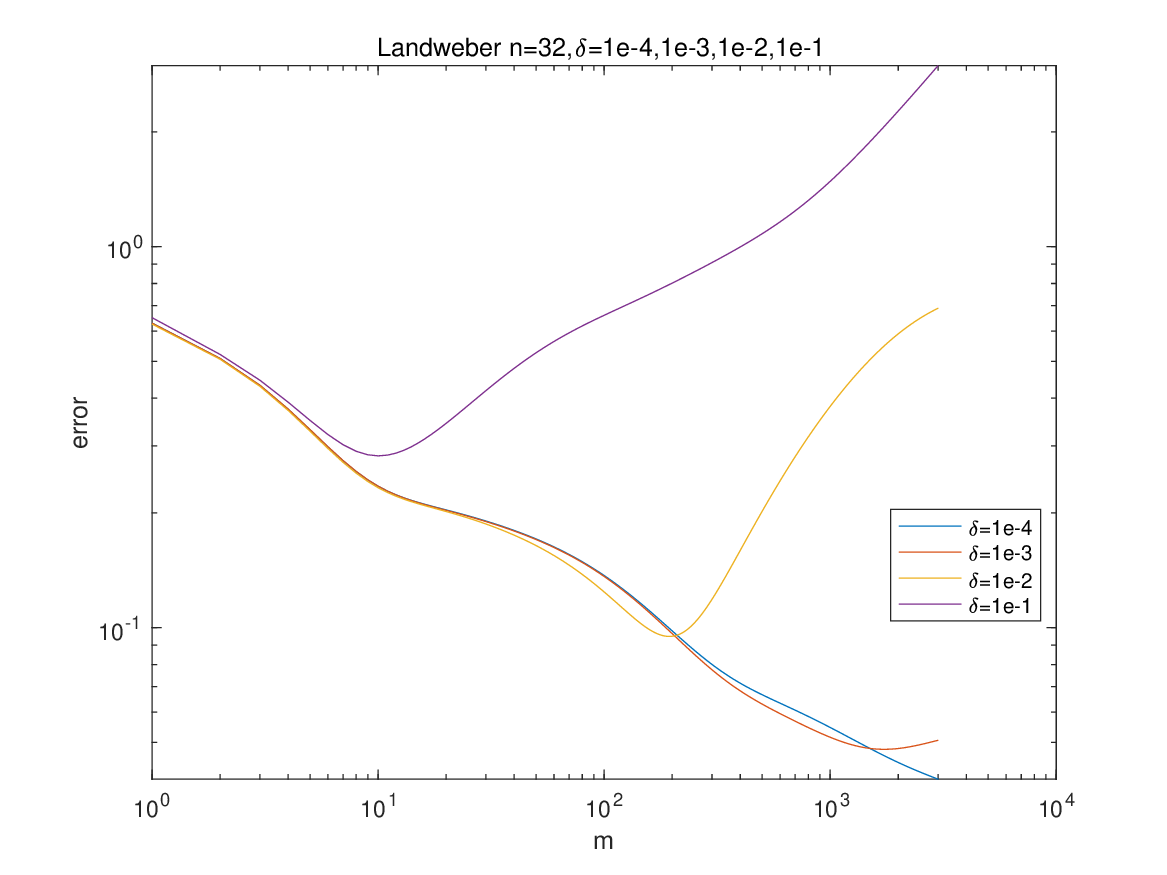}}\hfill
	 	\subfigure[The relationship between residual norm and solution norm for different perturbation parameter $\delta$ in Tikhonov]{\label{fig:residual_solution_delta_n32_one}\includegraphics[width=0.45\textwidth]{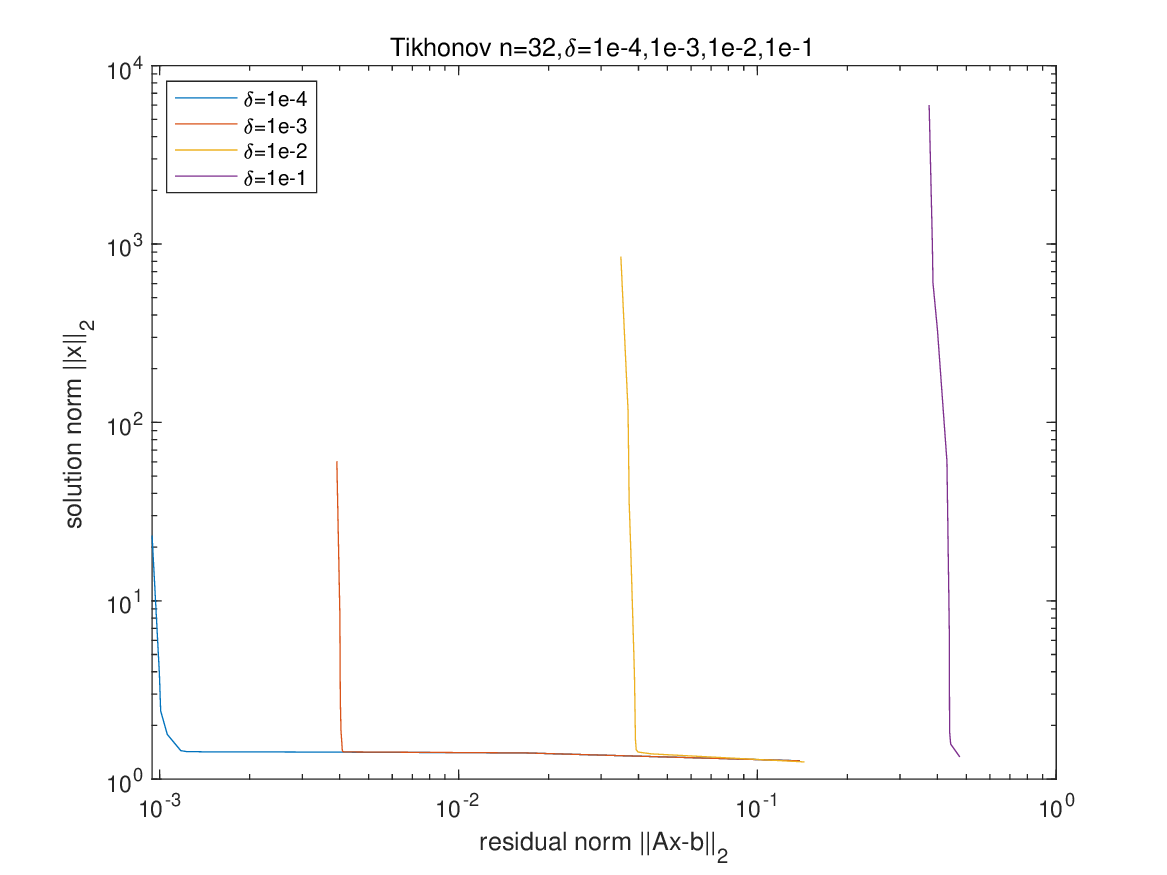}}\hfill
\caption{Numerical error for $n$ and $\delta$ in Tikhonov  and Landweber method}
\label{fig1}
\end{figure}

	 Next,  we consider the numerical error for different perturbation $\delta$ in Tikhonov method  (see Figure~\ref{fig:alpha_error_n32_one}),    and in Landweber method with $a=0.5$ to solve (\ref{eq:Fredholm1}) and iteration steps $m=1,2,...,3000$  (see Figure~\ref{fig:Landerr_m_n32_one}).  From the trends of the figures, they show that the numerical error  first decrease then increase as $\alpha$ or $m$ increase, this  coincides with the theory. Besides, we observe that both methods are comparable where
	 precision is concerned.

	   Figure~\ref{fig:residual_solution_delta_n32_one}  presents the relationship of residual norm and solution norm, when the magnitude of perturbation $\delta=10^{-j},j=1,2,3,4$ in Tikhonov method. As we can see, the small perturbation will have a small error with the same $\alpha$ basically. Besides, it looks like a $L$ curve, that is  to say, there is a  optimal $\alpha$ keeping the solution norm and residual norm balance. 
	 \subsection{New iterated Tikhonov  implementation}
	 Now we use the new iterated Tikhonov regularization method to solve (\ref{eq:Fredholm1}). Let $\alpha=1e-3$, $a=0.5$,  $\delta=1e-4$, $l=4$, and $n=32$, then we compare the total numeical error by using  classical Tikhonov method (\ref{eq:iterTik}),  Landweber iteration method  (\ref{eq:Land}) and the new  iterated Tikhonov method   (\ref{eq:newiter})  when iteration steps changes, see Figure~\ref{fig:newfilter}. 
	 \begin{figure}[!htbp]
	 	\centering
	 	\includegraphics[width=0.5\textwidth]{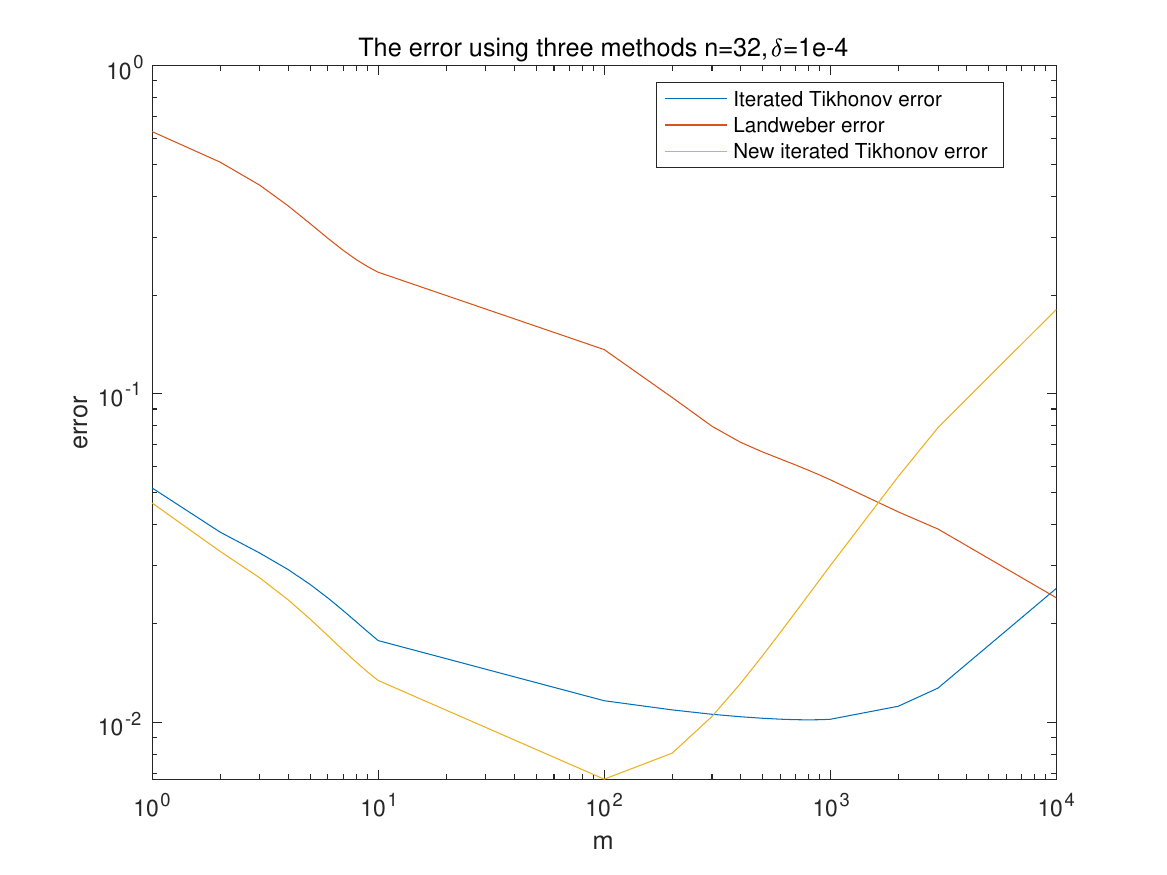}
	 	\caption{Iterated Tikhonov error, Landweber  iteration error  and New iterated Tikhonov error for different $m$} \label{fig:newfilter}
	 \end{figure}
	 From Figure~\ref{fig:newfilter}, we find that the new iterated Tikhonov method only need less iteration steps to get a  smaller error than the other two methods for this problem under these parameters setting, which proves the validity of the proposed method. 
	 Finally, let $l=2,m=100,r=0.8$, $\alpha=1e-0,9*1e-1,1e-3,1e-3$ for different perturbation $\delta$, the following Table~\ref{two} gives the auxiliary specification to prove the efficiency of the new iterated Tikhonov  method.
 \begin{table}[!tbh]
	     \centering
	     \caption{The  numerical error for different $\delta$ by three methods} \label{two}
	 \medskip\small\renewcommand{\arraystretch}{1.15}
	     \begin{tabular}{||c|cccc||}
	    \hline
	           & $\delta=1e-4$ & $\delta=1e-3$ &$\delta=1e-2$&$\delta=1e-1$\\
	         \hline
	         Iterated Tikhonov & 0.2692 & 0.1035&0.0648&0.0051 \\
	        
	        Landweber & 0.8654 & 0.1598 &0.1336&0.1370\\
	       
	         New iterated Tikhonov & 0.2418 & 0.0734 &0.0222&0.0046\\
	         \hline
	     \end{tabular}
	 \end{table}
	 \section{Conclusion}\label{sec5}
	 This paper has shown the iterated fractional weight regularization method is an efficient method to solve the Fredholm integral equation of the first kind. The numerical experiments conducted have validated the accuracy of the proposed method and shown that the comparability with the  classical  Tikhonov method. 

\bigskip


\end{document}